\documentclass[12pt]{article}
\usepackage[hidelinks]{hyperref}
\usepackage{float}
\usepackage[version = 4]{mhchem}
\usepackage{pdfpages}
\usepackage{mathrsfs}
\usepackage{longtable}
\usepackage{listings}
\usepackage{textcomp}
\usepackage{cancel}
\usepackage{enumerate}
\usepackage{amsmath}
\usepackage{amsthm}
\usepackage{amssymb}
\usepackage{booktabs}
\usepackage{colortbl}
\usepackage{graphicx}
\usepackage{subfigure}
\usepackage{ulem}
\usepackage{siunitx}
\usepackage{appendix}
\usepackage{wrapfig}
\usepackage{geometry}
\usepackage{indentfirst}
\usepackage{blkarray}
\usepackage{multirow} 
\usepackage{amsmath}
\usepackage{amssymb}
\usepackage{geometry}

\usepackage{authblk}

\usepackage[
backend=biber,
style=alphabetic,
sorting=ynt
]{biblatex}

\usepackage{JI_MathCourse_Notations}
\newtheorem{thm}{Theorem}[section]
\newtheorem{defn}{Definition}[section]
\newtheorem{lemma}{Lemma}[section]

\newtheorem{cor}{Corollary}[section]

\title{\textbf{Gradient Flow in Morse and Analytic Settings}}
\author[1]{Leyang Zhang \\ email \href{mailto:leyangz_hawk@outlook.com}{leyangz\_hawk@outlook.com}}
\affil[1]{Department of Mathematics, College of Liberal Arts \& Sciences, University of Illinois Urbana-Champaign}

\date{}

\begin{document}
\maketitle

\begin{abstract}
    This report is actually the in-class project of MATH489 FA2022, \textit{Dynamics and Differential Equations} in University of Illinois Urbana-Champaign. The report is mainly about gradient flow of a Morse-type functions (Morse functions, Morse--Bott functions, etc.) and analytic functions. Several small results that I develop during the reading are included as well. They have a focus on the convergence, limiting rate and direction of gradient flow in either setting. 
\end{abstract}

\section{Introduction}

Given a differentiable function $f: \bR^n \to \bR$, the gradient vector of $f$ at each $p \in \bR^n$, $\nabla f(p)$, points at the direction of the fastest value change. More generally, this can be formalized on a Riemannian manifold $(M, g)$, where $g$ is the Riemannian metric. If $f$ is smooth, the gradient vectors on $M$ form a (smooth) vector field to which we can define a flow. This is called the gradient flow of $f$ (more formal definition will be given in the next section). \\

Gradient vector fields and the gradient flow have received studies both in Mathematics and engineering, where various special properties have been discovered. These properties make gradient flow a widely used method in optimization for finding (local or global) minimum of a function. In this report, we will focus on gradient flows in Morse-type function and analytic function settings. In section \ref{section 2}, I make the basic definitions and assumptions. In the following section, I list the Morse theories I read and present several corollaries I made from them. In section \ref{section 4} , I list results related to gradient flows of Morse-Smale and Morse--Bott functions, as well as results I made when reading about them; proofs will be given. All these results are based on section \ref{section 3}. In section \ref{section 5}, I focus instead on functions satisfying Lojasiewicz inequality, which include (real) analytic and generalized Morse--Bott functions. Again, results I read and I developed are included. \\

All the results \textit{with proofs} are what I came up with during my reading. It is believed that some of them may already exist; I still included them as the proofs could be different, and I hope that an alternative proof may help the reader gain some new understandings to these results.  

\section{Definition and Assumption}\label{section 2} 

We make the following definitions and assumptions throughout the report. Let $M$ be an $n$-dimensional (smooth) manifold. We say $f: M \to \bR$ is a smooth map if all orders of its derivatives exist and are continuous. $p\in M$ is called a critical point if all partial derivatives of $f$ vanish at $p$. Moreover, given $p\in M$ and a chart $(U, (x_i)_{i=1}^n)$, $\text{Hess} f(p) := \left( \frac{\partial^2 f}{\partial x_j \partial x_j} |_p \right)_{i,j}$ is the Hessian of $f$. We are interested in non-degenerate critical points which we define below. 

\begin{defn}
    Let $p \in M$ be a critical point of $f$. $p$ is called non-degenerate if $\text{Hess} f(p)$ is non-singular. 
\end{defn} 

Let $A$ be a bilinear functional. we say the index of $A$, $\text{ind}(A)$ is the maximal dimension of subspaces on which it is negative definite, and the nullity of $A$ is the dimension of its kernel. Thus, a critical point $p$ of $f$ s non-degenerate if the nullity of $\text{Hess} f(p)$ is 0. Also, we define the index of a critical point $p$ of $f$ to be $\text{ind}(\text{Hess} f(p))$. If $\text{ind}(\text{Hess} f(p)) = 0$, $p$ is called a local minimum of $f$, if $\text{ind}(\text{Hess} f(p)) \in (0, n)$, $p$ is called a saddle and if $\text{ind}(\text{Hess} f(p)) = n$, $p$ is called a local maximum of $f$. 

\begin{defn}
    Let $f: M \to \bR$ be smooth. $f$ is called a Morse function if all its critical points are non-degenerate. 
\end{defn}

Let $(M, g)$ be a Riemannian manifold with Riemannian metric $g$. Consider a smooth function $f: M\to \bR$. We are interested in the gradient flow (GF) of $f$, which is a dynamical system described by 
\begin{equation}\label{GF defn}
    \gamma: (a,b) \to M, \parf{d\gamma}{dt} + \nabla f(\gamma(t)) = 0
\end{equation}
where $-\infty \leq a < b \leq \infty$ and $\nabla f$ is the gradient vector field on $M$ determined by 
\begin{equation}
    \< \nabla f(p), v \> = d_p f(v) \quad \quad \forall\, v \in T_p M, \forall\, p\in M. 
\end{equation}
We write $\gamma$ for both the function and its image and $\gamma_p$ to be the solution to (\ref{GF defn}) with initial value $p \in M$. 

\textbf{Remark. } When we discuss GF on $M$, we will assume that $M$ has a Riemannian metric $g$. 

\begin{defn}
    Let $f: M\to \bR$ be smooth and consider the system (\ref{GF defn}). Given any critical point $p \in M$, define 
    \begin{equation}
        W^s (p) := \{x \in M: \lim_{t\to\infty} \gamma_x(t) = p \} 
    \end{equation}
    to be the stable manifold of $p$ and 
    \begin{equation}
        W^u (p) := \{x \in M: \lim_{t\to-\infty} \gamma_x(t) = p\}
    \end{equation} 
    to be the unstable manifold of $p$. 
\end{defn}

\section{Basic Morse Theory}\label{section 3} 

We begin with the Morse lemma, which states that locally near a non-degenerate critical point of a smooth $f: M \to \bR$, $f$ ``looks like" a quadratic function. 

\begin{lemma}[Morse Lemma] Let $f: M \to \bR$ be smooth. Let $p\in M$ be a non-degenerate critical point of $f$ with $\text{ind}(\text{Hess} f(p)) = \lambda$. There is a coordinate system $(U, \varphi = (x_i)_{i=1}^n)$ centered at $p$, such that 
\begin{equation*}
    f \circ \varphi(x_1, ..., x_n) = f(p) + \left( -\sum_{i=1}^\lambda y_i^2 \right) + \left( \sum_{i=\lambda+1}^n y_i^2 \right). 
\end{equation*}
\cite{Milnor}
\end{lemma}

Due to this result, we can immediately obtain the following results. 

\begin{cor}
    Let $f: M \to \bR$ be smooth. The following results hold. 
    \begin{itemize}
        \item [(a)] The non-degenerate critical points of $f$ are isolated. 

        \item [(b)] The set of non-degenerate critical points of $f$ is locally finite. In particular, if $M$ is compact, there are finitely many non-degenerate critical points of $f$. 

        \item [(c)] If $M$ is compact and $f$ is a Morse function, then $f$ has only finitely many critical points. 
    \end{itemize}
    \cite{Milnor}
\end{cor}

\begin{thm}
    Let $f: M\to \bR$ be smooth and $p \in M$ be non-degenerate with index $\lambda$. Let $c = f(p)$. Suppose that for some $\varepsilon_0 > 0$, $f^{-1} [c-\vep_0, c+\vep_0]$ is compact and contains no other critical points of $f$, then for sufficiently small $\vep < \vep_0$, $f^{-1}(-\infty, c+\vep]$ has the homotopy type of $f^{-1}(-\infty, c-\vep]$ with a $\lambda$-cell attached \cite{Milnor}.  
\end{thm}

The next two theories discuss how the geometry of level sets of $f$ changes after passing a non-degenerate critical point, at least locally. 

\begin{thm}[Morse Theory] 
    Let $f: M\to \bR$ be a Morse function and for each $x \in \bR$, $f^{-1}(-\infty, x]$ is compact. Then $M$ has the homotopy type of a CW complex, with one cell of dimension $\lambda$ for each critical point of index $\lambda$ \cite{Milnor}. 
\end{thm}

For Morse functions, the geometry of stable and unstable manifolds of a critical point of $f$ (with respect to GF) is nice. 

\begin{thm}[Stable Manifold Theorem] 
    Let $M$ be a manifold of dimension $n$ and $f: M\to \bR$ be a Morse function. For any critical point $p$ of $M$ with index $\lambda$, $W^s(p)$ and $W^u(p)$ are smooth submanifolds diffeomorphic to $(n-\lambda)$-dimensional and $\lambda$-dimensional open balls, respectively \cite{Cohen}. 
\end{thm}

The next three corollaries state properties of GF of a Morse function $f$. The first one exploits the hyperbolicity of each critical point of a Morse function. The second one states that almost every GF of a Morse function on a compact manifold converges to a local minimum. The next one characterizes the behavior of $f$ when restricted to the stable and unstable manifolds of a critical point. The last one states that the lengths of all such GFs are bounded uniformly by a constant. 

\begin{cor}\label{Cor: enter once}
    Let $f: M \to \bR$ be a Morse function. For any critical point $p$ of $f$, there is an open $U \siq M$ such that any GF of $f$ does not enter $U$ twice \cite{Cohen}. 
\end{cor}

Then we have the following short result on convergence of gradient flow of $f$ on a compact manifold. 

\begin{cor}
    Let $M$ be a compact manifold and $f: M \to \bR$ be a Morse function. Then for almost all $p \in M$, the gradient flow $\gamma_p$ with initial value $p$ converges to a local minimum of $f$ as $t \to \infty$. 
\end{cor}
\begin{proof}
    Since $M$ is compact, $f$ has only finitely many critical points, say $p_1, ..., p_k \in M$. For simplicity, assume that the saddles and local maximums are $p_1, ..., p_l$, and the local minimums are $p_{l+1}, ..., p_k$. Let $U_1, ..., U_k$ be the open balls around $p_1, ..., p_k$, respectively, each of them satisfying Corollary \ref{Cor: enter once}. 

    For any GF $\gamma$, we claim that $\gamma$ converges to some $p_j$. Indeed, $\gamma$ has an accumulation point because its image is contained in the compact $M$. It cannot have two accumulation points, since otherwise $\gamma$ will enter some $U_j$ more than once, contradicting Corollary \ref{Cor: enter once}. This means $\limftyt \gamma(t) =: p$ exists and $\nabla f(p) = \limftyt \nabla f(\gamma(t)) = 0$. Thus, $p = p_j$ for some $j \in \{1, ..., k\}$, proving our claim. 

    Now if $\limftyt \gamma(t) \in \{p_1, ..., p_l\}$, $\gamma(0) \in \cup_{j=1}^l W^s(p_j)$, which is a measure-zero subset of $M$. For any $x_0 \notin \cup_{j=1}^l W^s(p_j)$, $\limftyt \gamma(t) \in \{p_{l+1}, ..., p_k \}$, as desired. 
\end{proof}

\begin{cor}
    Let $f: M\to \bR$ be a Morse function and $p$ be a critical point of $f$. Then $p$ is the unique minimum of $f|_{W^s(p)}: W^s(p) \to \bR$, the restriction of $f$ to $W^s(p)$; similarly, $p$ is the unique maximum of $f|_{W^u(p)}: W^u(p) \to \bR$, the restriction of $f$ to $W^u(p)$ \cite{Cohen}. 
\end{cor}

\begin{defn}[Nice metric]
    Let $(M, g)$ be a Riemannian manifold and $f$ be a Morse function. A metric is said to be nice if there exist charts around each critical point of $f$ so that the GF represented in any one of the charts has the form
    \begin{equation}
        \dot{x}_i = c_i x_i, \quad 1 \leq i \leq n, 
    \end{equation}
    where $c_i \in \bR\cut\{0\}$ for each $1 \leq i \leq n$ \cite{Cohen}. 
\end{defn}

\begin{cor}
    Let $(M,g)$ be a compact Riemannian manifold and $f: M \to \bR$ be a Morse function. Suppose that $g$ is a nice metric. Then there is a $C > 0$ such that the length of any GF of $f$ is bounded (above) by $C$ \cite{Cohen}. 
\end{cor}

\section{GF of Morse-Smale and Morse--Bott Functions}\label{section 4} 

In this section, we first focus on Morse-Smale functions, which is a special kind of Morse function whose stable and unstable manifolds intersect transversely. This transversality makes GF of Morse-Smale functions nice. Then we focus on Morse--Bott functions, which is a generalization of Morse function. Many properties of Morse function can be generalized to Morse--Bott functions. 

\begin{defn}[Morse-Smale function]
    Let $f: M \to \bR$ be smooth. $f$ is called a Morse-Smale function if it is a Morse function and for any critical points $p, q$ of it, $W^u(p)$ and $W^s(q)$ intersect transversely. In this case, denote $W(p,q) := W^u(p) \cap W^s(q)$. 
\end{defn} 
\textbf{Remark. } For any critical point $p \in M$, $W(p,p) = \{p\}$. Indeed, for any $x \in W(p,p)$, we would have 
\begin{equation*}
    f(p) = \lim_{t \to \infty} f(\gamma_x(t)) \leq f(x) \leq \lim_{t \to -\infty} f(\gamma_x(t)) = f(p). 
\end{equation*} 
Since $\gamma_x$ is a GF, $f(\gamma_x(\cdot))$ is monotonic and thus it is constant. But this means $\nabla f(\gamma_x(t)) = 0$ and $\gamma_x (t) = p$ for all $t$. In general, given critical points $p, q$ of $f$, since $W^u(p)$ and $W^s(p)$ intersects transversely, $W(p,q)$ is a smooth manifold with
\begin{equation*}
    \dim W(p,q) = \dim(W^u(p)) + \dim(W^s(q)) - n
\end{equation*}

\begin{defn}
    Given a Morse-Smale function $f: M\to \bR$ and critical points $p,q$ of it. For $t \in \bR$ between $f(p)$ and $f(q)$, define $W^t(p,q) := W(p,q) \cap f^{-1}(t)$. 
\end{defn}

The transversality condition implies that each $W(p,q)$ looks like a ``tube” $W^t(p,q) \times \bR$. 

\begin{thm}\label{W^t(p,q) of Morse-Smale}
    If the critical points $p,q$ satisfy $f(q) < f(p)$, for any $t \in (f(q), f(p))$, $W^t(p,q)$ is a submanifold of $W(p, q)$ with codimension 1. Moreover, the ``flow map" 
    \begin{equation*}
        \phi: W^t(p,q) \times \bR \to W(p,q), \phi (x, s) = \gamma_x(s) 
    \end{equation*} 
    is a diffeomorphism \cite{Cohen}. 
\end{thm}

\begin{defn}
    Let $f: M\to \bR$ be Morse-Smale and $a, b \in M$ be distinct critical points of it. Define 
    \begin{equation}
    \begin{aligned}
        P(a, b) &:= \{\eta \in C^1(\bR, M): \lim_{t\to-\infty} \eta(t) = a, \limftyt \eta(t) = b\}; \\
        F(a, b) &:= \{\eta \in P(a, b):  \dot{\eta}(t) = -\nabla f(\eta(t)) \}. 
    \end{aligned} 
    \end{equation}
    Also define $L: P(a,b) \to C^\infty (\bR, TM)$ by $L(\eta)(t) = \dot{\eta}(t) + \nabla f(\eta(t))$ for any $\eta \in P(a,b)$. 
\end{defn}

The following result describes when a Morse function is Morse-Smale. 

\begin{thm}
    Suppose that $(M, g)$ is a Riemannian manifold and let $f: M\to \bR$ be Morse. Suppose that $g$ is a nice metric. Then for any distinct critical points $p, q$ of $f$, $W^u(p)$ intersects transversely with $W^s(q)$ if and only if $dL$ is surjective at $F(p,q)$ \cite{Cohen}. 
\end{thm}

Then I present the results about Morse--Bott functions. 

\begin{defn}
    Let $f: M \to \bR$ be smooth and $N \siq M$ be a critical submanifold of $f$. Let $v(N) \to N$ be the normal bundle of $N$. Given $x \in N$, let $\text{Hess}_x^N (f): v_x(N) \times v_x(N) \to \bR$ be the restriction of $\text{Hess}_x f$ to $v_x(N)$. $N$ is called a degenerate critical manifold if $\text{Hess}_x^N (f)$ is non-degenerate for each $x \in N$ and $f$ is called Morse--Bott if its critical points are disjoint union of non-degenerate critical manifolds. 
\end{defn}
\textbf{Remark. } Intuitively, a Morse--Bott function is a smooth function whose cross sections along each of its critical manifolds are Morse. We can also view a Morse--Bott function as a family of parametrized Morse functions. Thus, similar as in the Morse function case, we may define the index of a connected $N$ to be $\text{ind}(\text{Hess}_x^N (f))$ for any $x \in N$. Note that the index map is continuous, so it does not change value on $N$. \\

Many results about Morse functions can be generalized to Morse--Bott functions. The first one says that a Morse--Bott function looks like a quadratic function along the normal of its critical manifolds. The second is about the stable and unstable manifolds of a critical manifold of it; which ``divides" the domain of the function in to different parts with respect to limiting behavior of GF. Conversely, one may ask what points in a critical manifold can be the limit of a non-stationary GF. This is characterized by Corollary \ref{Cor Every point in N is a limit of a GF}. Finally, we show that if a GF $\gamma$ converges to some $p$ in a critical manifold of a Morse--Bott function, then it is ``biased" towards the normal space at $p$, in the sense that the tangential change of $\gamma$ can be eventually bounded by the square of its normal change (Corollary \ref{Cor Normal bias of GF for Morse--Bott func}). As a result, if $p$ is a local minimum then $\gamma$ converges to $p$ exponentially fast. 

\begin{thm}
    Let $f: M\to \bR$ be Morse--Bott and $N$ a $l$-dimensional critical submanifold of $f$ with index $\lambda$. For any $x \in N$, there are open $U \siq N$ and open $V \siq M$, both containing $x$, and a diffeomorphism $\varphi: U \times \bR^\lambda \times \bR^{n-l-\lambda} \to V$ such that 
    \begin{equation*}
        f \circ \varphi(u; (v_i)_{i=1}^\lambda ; (w_i)_{i=1}^{n-l-\lambda}) = f(x) + \left( -\sum_{i=1}^\lambda v_i^2 \right) + \left( \sum_{i=1}^{n-l-\lambda} w_i^2 \right). 
    \end{equation*} 
    \cite{Banyaga}\cite{Cohen}
\end{thm}

Let $N$ be a non-degenerate critical manifold of $f$. Analogous to the case of Morse functions, we can define the stable and unstable manifolds of $N$, 
\begin{equation}
\begin{aligned}
    W^s(N) &:= \{x \in M: \limftyt \gamma_x(t) \in N\}; \\ 
    W^u(N) &:= \{x \in M: \lim_{t\to-\infty} \gamma_x(t) \in N \}. 
\end{aligned}
\end{equation} 
Also, for critical manifolds $N_1, N_2$ of a Morse--Bott function $f$, define $W(N_1, N_2) := W^u(N_1) \cap W^s(N_2)$. With these notations, Theorem \ref{W^t(p,q) of Morse-Smale} can be generalized as follows. 

\begin{thm}
    Let $f: M\to \bR$ be a Morse--Bott function. Suppose that for any connected critical manifolds $N_1, N_2$ of $f$, $W^u(N_1)$ and $W^s(N_2)$ intersect transversely, then the following results hold.
    \begin{itemize}
        \item [(a)] $W(N_1, N_2)$ is a smooth manifold with 
        \begin{equation*}
            \dim (W(N_1, N_2)) = n_1 + \lambda_1 - \lambda_2, 
        \end{equation*} 
        where $n_1 = \dim N_1$, $\lambda_1$ is the index of $N_1$ and $\lambda_2$ is the index of $N_2$. 
        \item [(b)] For any $t$ between $f(N_1)$ and $f(N_2)$, $W^t(N_1, N_2) := W(N_1, N_2) \cap f^{-1}(t)$ is a submanifold of $W(N_1, N_2)$ of codimension 1. 
        \item [(c)] The map $\phi: W^t(N_1, N_2) \times \bR \to W(N_1, N_2)$, $\phi(p, s) = \gamma_p(s)$ is a diffeomorphism. 
    \end{itemize}
\end{thm}
\begin{proof}
\begin{itemize}
    \item [(a)] Just note that $W^u(N_1)$ is a submanifold of $M$ with dimension $\dim N_1 + \lambda_1$ and $W^s(N_2)$ is a submanifold of $M$ with dimension $\dim N_2 + ((n - \dim N_2) - \lambda_2) = n - \lambda_2$. Also note that by hypothesis, $W^u(N_1)$ and $W^s(N_2)$ intersect transversely. 

    \item [(b)] For any $p \in W(N_1, N_2) \cap f^{-1}(t)$, the GF $\gamma_p$ satisfies $\limftyt \gamma_p(t) \in N_2$ because $p \in W^s(N_2)$. Therefore, $\nabla f(p) \neq 0$. Equivalently, $d_p f$ is surjective at each $p \in W(N_1, N_2) \cap f^{-1}(t)$. Thus, $W^t(N_1, N_2) = W(N_1, N_2) \cap f^{-1}(t)$ is a submanifold of $W(N_1, N_2)$ with codimension 1. 

    \item [(c)] By the uniqueness of solution to the ODE 
    \begin{equation*}
        \dot{\gamma}(s) = - \nabla f(\gamma(s)), 
    \end{equation*} 
    we deduce that $\phi$ is injective. Since $\nabla f$ is smooth, the solution to the ODE depends smoothly on initial condition and time $s$. Thus, $\phi$ is smooth. Given $q \in W(N_1, N_2)$, consider the GF $\gamma_q: \bR \to W(N_1, N_2)$. Since 
    \begin{equation*}
        \limftyt \gamma_q(t) \in N_2, \quad \quad \lim_{t \to -\infty} \gamma_q(t) \in N_1, 
    \end{equation*} 
    and since $f\circ \gamma_q$ is monotonic, there exist some $s \in \bR$ with $\gamma_q(-s) := p \in W^t(N_1, N_2)$. Then $\gamma_p (s) = \gamma_q(s - s) = q$. Thus, $\phi$ is surjective. 

    Finally, we will show that for any $(p,s) \in W^t(N_1, N_2) \times \bR$, $d_{(p,s)} \phi$ has rank equal to $\dim W(N_1, N_2)$. Since $\dim (W^t(N_1, N_2) \times \bR) = \dim W(N_1, N_2)$, this would imply that $\phi$ is a bijective local diffeomorphism and thus a diffeomorphism. First consider $s = 0$. Note that the map $W^t(N_1, N_2) \ni p \mapsto \phi(p, 0)$ is just the inclusion map from $W^t(N_1, N_2)$ to $\phi(p,0)$. Thus, the rank of its differential at $p$ is 
    \begin{equation*}
        \dim W^t(N_1, N_2) = \dim W(N_1, N_2) - 1. 
    \end{equation*} 
    Given $s \in \bR$, consider the diffeomorphism 
    \begin{equation*}
        \psi_s: W(N_1, N_2) \to W(N_1, N_2), \psi_s(q) = \gamma_q(s). 
    \end{equation*} 
    Clearly, 
    \begin{equation*}
        \text{rank}\, \partial_p \phi(p,s) = \text{rank}\, d_p(\psi_s \circ \phi(p,0)) = \dim W^t(N_1, N_2), 
    \end{equation*} 
    where $\partial_p \phi(p,s)$ denotes the derivatives of $\phi$ with respect to the $p$-space $W^t(N_1, N_2)$ at the point $(p,s)$. We need only show that $\partial_s \phi (p,s)$ is not in the range of $\partial_p \phi (p,s)$. 

    Let $N = \psi_s(W^t(N_1, N_2))$. On one hand, we have 
    \begin{equation*}
        \partial_s \phi(p,s) \left( T_{(p,0)} W^t(N_1, N_2) \times \{0\} \right) \siq \text{ran} (d_p \psi_s (W^t(N_1, N_2))) \siq T_{\gamma_p(s)} N. 
    \end{equation*} 
    On the other hand, using 
    \begin{equation*}
        \frac{d}{d\zeta} \psi_s \circ \phi(p,s) = \frac{d}{d\zeta} \gamma_{\phi(p,\zeta)}(s) = \frac{d}{d\zeta} \gamma_p(s+\zeta), 
    \end{equation*} 
    we can see that 
    \begin{equation*}
        \dot{\gamma}_p(s) = \frac{d}{d\zeta}(\psi_s \circ \phi(p,\cdot))|_{\zeta = 0} = d_p\psi_s \circ \dot{\gamma}_p(0) = d_p \psi_s \circ \nabla f(p). 
    \end{equation*}
    Since $\nabla f(p) \notin T_p W^t(N_1, N_2)$ and since $\psi_s$ is a diffeomorphism, $\dot{\gamma}_p(s) \notin T_{\gamma_p(s)} N$. Therefore, 
    \begin{align*}
        \text{rank}\, d_{(p,s)} \phi 
        &= \text{rank}\, \partial_p \phi(p,s) + \text{rank}\, \partial_s \phi(p,s) \\ 
        &= \dim W^t(N_1, N_2) + 1 \\
        &= \dim W(N_1, N_2). 
    \end{align*} 
    This implies that $\phi$ is a local diffeomorphism and completes the proof. 
\end{itemize}
\end{proof}

\begin{cor}\label{Cor Every point in N is a limit of a GF}
    Let $f: \bR^n \to \bR$ be a non-negative Morse--Bott function. Suppose that $f^{-1}(0)$ is a non-empty submanifold of $\bR^n$. Then for any $p \in N$, there is a non-constant GF $\gamma$ with $\limftyt \gamma(t) = p$.
\end{cor} 
\begin{proof}
    This is a special case of Corollary \ref{Cor: limit point correspondence} which we will prove in the next section. 
\end{proof}

\begin{cor}\label{Cor Normal bias of GF for Morse--Bott func}
    Let $f: \bR^n \to \bR$ be a non-negative Morse--Bott function. Suppose that $f^{-1}(0)$ is a non-empty submanifold of $\bR^n$. Given a GF $\gamma$ with $\limftyt \gamma(t) =: p \in f^{-1}(0)$, we have 
    \begin{equation*}
        \varlimsup_{t \to \infty} \frac{|\calP(\gamma(t) - p)|}{|\calQ(\gamma(t) - p)|^2} < \infty, 
    \end{equation*} 
    where $\calP$ is the orthogonal projection onto $T_p f^{-1}(0)$ and $\calQ = I - \calP$. 
\end{cor}
\begin{proof}
    The result is local, so by Theorem 4.3, we may for simplicity assume that $f^{-1}(0) = \{0\} \times \bR^{n-s} \siq \bR^n = \bR^s \times \bR^{n-s}$, where $s$ is the codimension of $f^{-1}(0)$. We may also write a point in $\bR^n$ as $(y,z)$, where $y \in \bR^s$ and $z \in \bR^{n-s}$. With these notations, on some compact neighborhood $U$ of $p$ the ODE for the GF $\gamma =: (y, z)$ takes the form 
    \begin{equation}
    \begin{aligned}
        \dot{y}(t) &= H(z(t)) y(t) + g(y(t), z(t)) \\ 
        \dot{z}(t) &= m(y(t), z(t)), 
    \end{aligned}   
    \end{equation}
     such that for any $(y,z) \in U$, $H(z)$ is negative definite, $H(z)$ changes smoothly with respect to $z$, $g(0,z) = m(0,z) = 0$, and $g(y,z) = O(|y|^2), \quad \quad m(y,z) = O(|y|^2)$. Therefore, we can also write 
     \begin{align*}
         \dot{y}(t) &= H(z(t)) y(t) + O(|y(t)|^2) \\ 
        \dot{z}(t) &= O(|y(t)|^2) 
     \end{align*} 
     for $t$ sufficiently large. 

    Denote $p = (0, z_p)$. What we need to show now is just 
    \begin{equation*}
        \varlimsup_{t \to \infty} \frac{|z(t) - z_p|}{|y(t)|^2} < \infty. 
    \end{equation*}
    Set 
    \begin{align*}
        \lambda &:= \inf \left\{ \frac{|\< H(z)h, h\>|}{|h|^2}: h \in \bR^s\cut \{0\}, (0,z) \in U \right\}; \\
        \lambda' &:= \sup \left\{ \frac{|\< H(z)h, h\>|}{|h|^2}: h \in \bR^s\cut \{0\}, (0,z) \in U \right\}; 
    \end{align*} 
    Since $\limftyt \gamma(t) = p$, $\limftyt g(\gamma(t)) = \limftyt m(\gamma(t)) = 0$. Thus, for each $1 \leq j \leq s$, $y_j(t)$ decreases at exponential rate. In particular, there is some $T > 0$ such that for any $\zeta \geq t > T$ and any $1 \leq j \leq s$, we have $|y_j(t)| < 1$, $\frac{\lambda}{2} |y_j(t)| \leq |\dot{y}_j(t)| \leq 2\lambda' |y_j(t)|$, and thus
    \[ D_1 |y_j(t)| e^{-\mu_1(\zeta - t)} \leq |y_j(\zeta)| \leq D_2 |y_j(t)| e^{-\mu_2(\zeta - t)} \] 
    for some $D_1, D_2, \mu_1, \mu_2 > 0$ depending only on $T$. 
    
    These observations imply that $y_j$ is decreasing and $y_j = 0$ or the sign of $y_j$ does not change on $(T, \infty)$. Thus, for any $t > T$ and any $1 \leq j \leq s$, 
    \begin{align*}
        \int_t^\infty y_j^2(\zeta) d\zeta 
        &\leq \sum_{k=\floor{t}}^\infty y_j^2(k) \\ 
        &\leq \left( \sum_{k=\floor{t}}^\infty y_j(k) \right)^2 \\ 
        &\leq \frac{4}{\lambda^2} \left( \int_{\floor{t}}^\infty \dot{y}_j(\zeta) d\zeta \right)^2 = \frac{4}{\lambda^2} y_j^2(\floor{t}), 
    \end{align*} 
    where $\floor{t}$ denotes the largest integer smaller than $t$. Since 
    \[ |y_j(t) \geq D_1 |y_j(\floor{t})|e^{-\mu_1(t-\floor{t})} \geq D_1 e^{-\mu_1} |y_j(\floor{t})|, \] 
    we have 
    \begin{equation*}
        |y_j(\floor{t})| \leq \frac{e^{\mu_1}}{D_1} |y_j(t)|. 
    \end{equation*} 
    Since $\dot{z}(t) = m(y(t), z(t)) = O(|y(t)|^2)$, there is some $C > 0$ such that $|\dot{z}(t)| \leq C|y(t)|^2$ for $t > T$. It follows that for $t > T$, 
    \begin{align*}
        |z(t) - z_p| \leq \int_t^\infty 
        &\leq C \sum_{j=1}^s \int_t^\infty y_j^2(\zeta) d\zeta \\ 
        &\leq C \frac{4}{\lambda^2} \frac{e^{\mu_1}}{D_1} \left( \sum_{j=1}^s y_j^2(t) \right). 
    \end{align*} 
    Letting $t \to \infty$, we see that 
    \begin{equation*}
        \varlimsup_{t \to \infty} \frac{|z(t) - z_p|}{|y(t)|^2} \leq C \frac{4}{\lambda^2} \frac{e^{\mu_1}}{D_1} < \infty. 
    \end{equation*}
\end{proof}
\textbf{Remark. } In \cite{Kurdyka} the authors show that if $f$ is analytic, the secant line 
\[
    \frac{\gamma(t) - \limftys \gamma(s)}{|\gamma(t) - \limftys \gamma(s)|}
\]
of any convergent GF $\gamma$ always has a limit. Thus, the result I give above further shows that the limit must be in the normal space of the critical manifold at the limit point. 

\section{Lojasiewicz inequality and Gradient Flow}\label{section 5}  

In the previous two sections, we discussed some properties of Morse and Morse--Bott functions, as well as the GF of them. Such function has an important property that its critical points are non-degenerate (or at least, non-degenerate in the normal space of the critical manifold for Morse--Bott function). However, not every gradient vector field has this property. For example, consider $f(x,y,z) = x^4 + y^2$. On a neighborhood of $0 \in \bR^3$ the GF converges to 0 and exhibits different behaviors from Morse or Morse--Bott functions. \\ 

Lojasiewicz inequality, which holds for (real) analytic functions, can be used to deal with functions with degenerate critical points. It is a generalization of  
\begin{equation}\label{Lojasiewicz ineq on real line}
    |f'(q)| \geq C|f(q) - f(p)|^\theta,  \quad \quad \exists\,\theta \in [1/2, 1) 
\end{equation} 
for a non-constant analytic function $f: \bR \to \bR$ near any one of its critical point $p$. Indeed, we have $f(q) = \sum_{i=0}^\infty a_i (q - p)^i$ whenever $q$ is in a neighborhood $U$ of $p$. If $i_0 \geq 2$ is the smallest positive integer with $a_{i_0} \neq 0$ then there are $A_1, A_2, B_1, B_2 > 0$ such that for any $p \in U$, 
\begin{align*}
    A_1 |p-q|^{i_0} &\leq |f(q) - f(p)| \leq A_2 |p-q|^{i_0}; \\ 
    B_1 |p-q|^{i_0} &\leq |f'(q) - 0| \leq B_2 |p-q|^{i_0}, 
\end{align*} 
from which we can easily see that equation (\ref{Lojasiewicz ineq on real line}) holds. 

\begin{thm}[Lojasiewicz inequality]
    Let $n \geq 1$ and $f: \bR^n \to \bR$ be a (real) analytic function. For any critical point $p$ of $f$, there is a neighborhood $U$ of $p$ and constants $C > 0$, $\theta \in [1/2, 1)$ such that 
    \begin{equation}\label{Lojasiewicz ineq}
        |\nabla f(q)| \geq C|f(q) - f(p)|^\theta
    \end{equation} 
    for any $q \in U$ \cite{Feehan}. 
\end{thm}
\textbf{Remark. } In fact, the inequality (\ref{Lojasiewicz ineq}) holds for generalized Morse--Bott functions at $p$ as well, which is a smooth function $g$ such that 
\begin{itemize}
    \item [(a)] $p$ is contained in a critical manifold $N$ of it. 
    \item [(b)] There is some $k \geq 2$ such that for any $j < k$, $D^j g(x) = 0$ for all $x \in N$ and $D^k g(p): \Pi_{j=1}^k \bR^n \to \bR$ is non-vanishing when restricted to $\Pi_{j=1}^k T_pN \siq \Pi_{j=1}^l \bR^n$. 
\end{itemize}
It is said that the original proof of inequality (\ref{Lojasiewicz ineq}) is based on techniques in analytic geometry \cite{Feehan}. Later \cite{Feehan} made an elementary, purely geometric and coordinate-based proof for generalized Morse--Bott functions. Moreover, for $f$ a non-negative analytic function and/or generalized Morse--Bott function, we have the following distance inequalities. 

\begin{cor}
    If $f(p) = 0$ and $Df(p) = 0$, there is a neighborhood $U$ of $p$ and constants $C > 0$ and $\alpha \geq 2$ such that 
    \begin{equation}
        f(x) \geq C \dist{x}{\calS}^\alpha, 
    \end{equation}
    where $\dist{y}{E} := \inf \{z \in E: |y - z|\}$ and $\calS$ is the set of critical points of $f$. In particular, if $\calS \siq f^{-1} (0)$ then 
    \begin{equation*}
        f(x) \geq C \dist{x}{f^{-1}(0)}^\alpha. 
    \end{equation*}
    \cite{Feehan}
\end{cor}

An important result of Lojasiewicz inequality is the local convergence of GF of non-negative analytic functions. 

\begin{thm}\label{Thm: GF convergence and curve length estimate}
    Let $f: \bR^n \to \bR$ be an analytic or generalized Morse--Bott function. Suppose that $p$ is a local minimum of $f$. Then $p$ has a neighborhood $U$ such that any non-constant GF with initial value $x_0 \in U$ converges as $t \to \infty$. Moreover, any such GF converges with some rate $0 < \beta < 1$ depending only on $p$, namely, 
    \[l(\gamma[t, \infty)) = O(|f(\gamma(t) - f(p)|^\beta)\] 
    as $t \to \infty$, where $l(\gamma[t, \infty))$ is the curve length of $\gamma[t, \infty)$ \cite{Absil}. 
\end{thm}
\textbf{Remark. } This in fact has a generalization to $C^1$ functions obeying Lojasiewicz inequality. It is proved in \cite{Feehan}

\begin{cor}
    Let $f: \bR^n \to \bR$ be an analytic function and let $p$ be a local minimum of $f$. Suppose that $\gamma: [0, \infty) \to \bR^n$ is a GF such that for some $\{t_j\}_{j=1}^\infty$ diverging to $\infty$, $\limftyj \gamma(t_j) = p$. Then $\gamma$ converges to $p$ as $t \to \infty$ \cite{Absil}. 
\end{cor}

I also did some characterization of analytic (and generalized Morse--Bott) functions, some of which are listed below. The first one is a summary of several results in \cite{Nowel}. I give a proof of it. 

\begin{thm}
    Let $f: \bR^n \to \bR$ be an analytic function and $p$ be a local minimum of $f$. For any $k$ sufficiently large, there is an open $U$ containing $p$ such that for any GF $\gamma$ with initial value in $U$, $\gamma$ intersects the set $Z := \{x \in \bR^n: f(x) - f(p) = |x - p|^{2k}\}$ at most once, and is transversal to it.
\end{thm}
\begin{proof}
    Since the result is local, we may work on a neighborhood $U$ of $p$ on which $f \geq f(p)$. First we show that for sufficiently large $k \in \bN$ and $|x - p|$ sufficiently small, when $x \in Z$ we have 
    \begin{equation*}
        \nabla f(x) \cdot \nabla (f(x) - |x - p|^{2k}) > 0. 
    \end{equation*} 
    By Lojasiewicz inequality, there is an open $U \siq V$ containing $x$ and some $C > 0$, $\theta \in [1/2, 1)$ such that for any $x \in U$, $|\nabla f(x)| \geq C |f(x) - f(p)|^\theta$. Choose $k$ so large that $2k\theta < 2k-1$. Then there is some open $U' \siq U$ containing $p$ such that for any $x \in U' \cap Z$, 
    \begin{equation*}
        |\nabla f(x)| \ge C|f(x) - f(p)|^\theta = C|x - p|^{2k\theta} > 2k |x - p|^{2k-1}. 
    \end{equation*}
    It follows that 
    \begin{align*}
        \nabla f(x) \cdot \nabla (f(x) - |x - p|^{2k}) 
        &= \nabla f(x) \cdot [\nabla f(x) - 2k |x - p|^{2k-2} (x-p)] \\ 
        &\geq \nabla f(x) \cdot [|\nabla f(x)| - 2k |x - p|^{2k-1}] \\ 
        &> 0, 
    \end{align*}
    proving our claim. 

    Now let $W \siq U'$ be a neighborhood of $p$ such that the image of any GF $\gamma: [0, \infty) \to \bR^n$ of $f$ with initial value in $W$ is contained in $U'$. For any such GF $\gamma$, note that the map 
    \begin{equation*}
        \tau: t \mapsto (f(\gamma(t)) - f(p)) - |\gamma(t) - p|^{2k}
    \end{equation*} 
    is smooth. Moreover, if $t \in \bR$ satisfies $\gamma(t) \in Z$, we have 
    \begin{align*}
        \dot{\tau}(t) 
        &= [\nabla f(\gamma(t)) - 2k|\gamma(t) - p|^{2k-2} (\gamma(t) - p)] \cdot \dot{\gamma}(t) \\ 
        &= - \nabla (f(\cdot) - 2k|(\cdot) - p|^{2k})|_{\gamma(t)} \cdot f(\gamma(t)) \\
        &< 0.
    \end{align*} 
    This implies that $\dot{\gamma}(t)$ does not lie in $T_{\gamma(t)} Z$ and $\tau(t) = 0$ for at most one $t \in \bR$. In other words, $\gamma$ intersects $Z$ at most once and the intersection is transversal. 
\end{proof}

Similar to the Morse--Bott setting, in analytic setting one may also ask what points in the critical set would be the limit of a GF. The following two results deal with this question. The first one, Lemma \ref{Lemma: dense limit set} can indeed be immediately obtained from Lojasiewicz inequality. Then I will use it to prove Corollary \ref{Cor: limit point correspondence}. 

\begin{lemma}\label{Lemma: dense limit set}
    Let $f: \bR^n \to \bR$ be a non-negative analytic or generalized Morse--Bott function. Suppose that $f^{-1}(0) \neq \emptyset$. Then the set
    \begin{equation*}
    \{p \in f^{-1}(0): \text{ there is a GF } \gamma \text{ with } \limftyt \gamma(t) = p \}  
    \end{equation*} 
    is dense in $f^{-1}(0)$. 
\end{lemma}
\begin{proof}
    Assume this is not true. Then there would be an open ball $V \siq \bR^n$ such that $V \cap f^{-1}(0)$ contains no limit of GF that belongs to $W^s(f^{-1}(0))$. On the other hand, since $f|_V: V \to \bR$, the restriction of $f$ to $V$, is a non-negative analytic function with $f|_V^{-1}(0) = V \cap f^{-1}(0)$, there is some $U \siq V$ such that for any $x_0 \in U$, the GF of $f|_V$ with initial value $x_0$ converges to a point in $f|_V^{-1}(0)$. This is a contradiction. 
\end{proof}

In fact, more can be said about the correspondence of convergent GF and points in the set of local minima of $f$. 

\begin{cor}\label{Cor: limit point correspondence}
    Let $f: \bR^n \to \bR$ be a non-negative analytic or generalized Morse--Bott function. Suppose that $N := f^{-1}(0)$ is a non-empty critical submanifold and each $p \in N$ has a neighborhood $U$ satisfying
    \begin{itemize}
        \item [(a)] For any $x_0 \in U$, $\gamma_{x_0}$ converges to a point in $N$. 
        \item [(b)] There are some $C, \alpha > 0$ such that for any $x_0 \in U$, the curve length of $\gamma_{x_0}$ is bounded (above) by $C \dist{x_0}{N}^\alpha$. 
    \end{itemize}
    Then for any $p \in N$, there is a non-constant GF converging to $p$ as $t \to \infty$. In particular, the result holds when $f$ is analytic or generalized Morse--Bott. 
\end{cor}
\begin{proof}
    By Lemma \ref{Lemma: dense limit set}, there is a sequence $\{p_j\}_{j=1}^\infty$ and a sequence of GF $\{\gamma_j\}_{j=1}^\infty$ such that $\limftyt \gamma_j(t) = p_j \in N$ and $\limftyj p_j = p$. Choose a compact neighborhood $V \siq U$ of $p$ (so $\bar{V} \siq U$) For each $j$, there is a largest $t_j \in \bR$ such that $x_j := \gamma_j(t_j) \in \partial V \cap \gamma_j$. Since $\partial V$ is compact, the sequence $\{x_j\}_{j=1}^\infty$ has an accumulation point $x$ in $\bar{V}$, Moreover, hypothesis (b) implies that $x \notin N$. Since $x \in U$, the GF $\gamma_x: [0, \infty) \to \bR^n$ converges to a point in $N$ and its curve length is bounded by $C \dist{x}{N}^\alpha$. 

    Let $l_j$ be the curve length of $\gamma_j[t_j, \infty)$ and $l$ be the curve length of $\gamma_x$. For each $j \in \bN$, define 
    \begin{align*}
        &u_j: [0, \infty) \to \bR^n, \\ 
        &\dot{u}_j(t) = - \frac{\nabla f(u_j(t)) \chi_{[0, l_j)}}{|\nabla f(u_j(t))|}, \quad u_j(0) = x_j
    \end{align*} 
    and similarly, define 
    \begin{align*}
        &u: [0, \infty) \to \bR^n, \\ 
        &\dot{u}(t) = - \frac{\nabla f(u(t)) \chi_{[0, l)}}{|\nabla f(u(t))|}, \quad u(0) = x. 
    \end{align*}
    In the formulas above, $\chi_E$ denotes the characteristic function of $E$, namely, $\chi_E(x) = 1$ for $x \in E$ and $\chi_E(x) = 0$ otherwise. Also note that the $u_j$'s and $u$ are exactly the trajectories of GFs. 

    Fix $\vep > 0$. Choose any $k \in \bN$ with $|x - x_k| < \vep$. There is some $T > 0$ such that for any $t > T$, we have $\dist{u_j(t)}{N} < \vep^{1/\alpha}$ and $\dist{u(t)}{N} < \vep^{1/\alpha}$. Then the Grownwall's inequality and hypothesis (b) yield 
    \begin{align*}
        \left|\limftys u(s) - \limftys u_k(s) \right| 
        &\leq |\limftys u(s) - u(t)| + |u(t) - u_k(t)| + |u_k(t) - p_k| \\ 
        &\leq C \dist{u(t)}{N}^\alpha + \exp(t) \vep + C \dist{u_k(t)}{N}^\alpha \\ 
        &\leq (2C + \exp(t)) \vep. 
    \end{align*} 
    Since $\{x\} \cup \{x_j\}_{j=1}^\infty$ is a subset of the bounded $\partial V$, it follows that $\sup \{l, l_1, l_2, ...\} < \infty$; in particular, there is some $T' > 0$ such that the $u_j$'s and $u$ are all constant on $[T', \infty)$. Thus, we actually have 
    \begin{equation*}
        \left|\limftys u(s) - \limftys u_k(s) \right| \leq (2C + \exp(T'))\vep
    \end{equation*} 
    Letting $\vep \to 0$, we see that $\limftys u(s) = \lim_{k \to \infty} p_k = p$, which means $u$, and thus $\gamma_x$, converges to $p$. This shows the first part of the corollary.  

    Now suppose that $f$ is analytic or generalized Morse--Bott. Let $p \in N$. By Theorem \ref{Thm: GF convergence and curve length estimate}, there is a bounded neighborhood $U$ of $p$ and some $\beta > 0$ such that for any $x_0 \in U$, $\gamma_{x_0}$ converges at rate $\beta$. Since $U$ is bounded and $f$ is smooth, $f$ is Lipschitz on $U$, so there is some $L > 0$ with $|f(x) - f(y)| \leq L|x - y|$ for any $x, y \in U$. It follows that the curve length of any such $\gamma_{x_0}$, $l(\gamma_{x_0})$, can be estimated by 
    \[ l(\gamma_{x_0}) \leq: \tilde{C} |f(x_0)|^\beta \leq |f(x_0) - 0|^\beta \leq \tilde{C}L^\beta |x_0 - q|^\beta, \]
    where $q \in N$ is any point satisfying $|x_0 - q| = \dist{x_0}{N}$ ($q$ exists because $N$ is closed). This shows $f$ satisfies the hypotheses and the desired result follows. 
\end{proof}

\printbibliography\label{section 6} 

\end{document}